\documentclass[12pt]{article}
\usepackage{amsmath}
\usepackage{amsfonts}
\usepackage{verbatim}
\usepackage{graphicx}
\usepackage{graphics}
\usepackage{epsfig}
\usepackage{amssymb}
\usepackage{amsthm}
\usepackage[usenames]{color}

\newtheorem{theorem}{Theorem}[subsection]
\newtheorem*{thma}{Remark}

\newtheorem{lemma}[theorem]{Lemma}
\newtheorem{definition}[theorem]{Definition}

\newtheorem{example}[theorem]{Example}
\newtheorem{remark}[theorem]{Remark}
\numberwithin{equation}{section}

\title{Crystal rules for $(\ell,0)$-JM partitions}

\author{Chris Berg\\
\small Fields Institute, Toronto, ON, Canada\\
\small \texttt{cberg@fields.utoronto.edu}\\
}

\begin{document}


\newlength\cellsize \setlength\cellsize{18\unitlength}
\savebox2{%
\begin{picture}(18,18)
\put(0,0){\line(1,0){18}}
\put(0,0){\line(0,1){18}}
\put(18,0){\line(0,1){18}}
\put(0,18){\line(1,0){18}}
\end{picture}}
\newcommand\cellify[1]{\def\thearg{#1}\def\nothing{}%
\ifx\thearg\nothing
\vrule width0pt height\cellsize depth0pt\else
\hbox to 0pt{\usebox2\hss}\fi%
\vbox to 18\unitlength{
\vss
\hbox to 18\unitlength{\hss$#1$\hss}
\vss}}
\newcommand\tableau[1]{\vtop{\let\\=\cr
\setlength\baselineskip{-16000pt}
\setlength\lineskiplimit{16000pt}
\setlength\lineskip{0pt}
\halign{&\cellify{##}\cr#1\crcr}}}
\savebox3{%
\begin{picture}(15,15)
\put(0,0){\line(1,0){15}}
\put(0,0){\line(0,1){15}}
\put(15,0){\line(0,1){15}}
\put(0,15){\line(1,0){15}}
\end{picture}}
\newcommand\expath[1]{%
\hbox to 0pt{\usebox3\hss}%
\vbox to 15\unitlength{
\vss
\hbox to 15\unitlength{\hss$#1$\hss}
\vss}}


\maketitle

\begin{abstract}

Vazirani and the author \cite{BV} gave a new interpretation of what we called $\ell$-partitions,
 also known as $(\ell,0)$-Carter partitions. The primary interpretation of such a partition
 $\lambda$ is that it corresponds to a Specht module $S^{\lambda}$ which remains irreducible
 over the finite Hecke algebra $H_n(q)$ when $q$ is specialized to a primitive $\ell^{th}$ root
 of unity. To accomplish this we relied heavily on the description of such a partition in terms of 
its hook lengths, a condition provided by James and Mathas. In this paper, I 
use a new description of the crystal $reg_\ell$ which helps extend previous results 
to all $(\ell,0)$-JM partitions (similar to $(\ell,0)$-Carter partitions, but
 not necessarily $\ell$-regular), by using an analogous condition for hook lengths which was proven by 
work of Lyle and Fayers.
\end{abstract}

\section{Introduction}

The main goal of this paper is to generalize results of \cite{BV} to a larger class of partitions. 
One model of the crystal $B(\Lambda_0)$ of $\widehat{\mathfrak{sl}_{\ell}}$, referred to here as $reg_\ell$, has as nodes 
$\ell$-regular partitions. In \cite{BV} we proved results about where on the crystal $reg_\ell$
a so-called $\ell$-partition could occur. $\ell$-partitions are the $\ell$-regular partitions for
which the Specht modules $S^{\lambda}$ are irreducible for the Hecke algebra $H_n(q)$ when $q$ is
specialized to a primitive $\ell^{th}$ root of unity. An $\ell$-regular partition $\lambda$ 
indexes a simple module $D^{\lambda}$ for $H_n(q)$ when $q$ is 
a primitive $\ell^{th}$ root of unity. We noticed that within the crystal  $reg_\ell$ that another type of partitions, which we call weak $\ell$-partitions, 
satisfied rules similar to the rules given in \cite{BV} for $\ell$-partitions. In order to prove this, 
we built an isomorphic version of the crystal $reg_\ell$, which we denote $ladd_\ell$. The description of
$ladd_\ell$, with the isomorphism to $reg_\ell$, can be found in \cite{B}.

\subsection{Summary of results from this paper}
In Section \ref{removing} we give a new way of characterizing $(\ell,0)$-JM partitions by their removable 
$\ell$-rim hooks. In Section \ref{construct_JM} we give a different characterization of $(\ell,0)$-JM partitions.
 Section \ref{extend} extends our crystal theorems from 
\cite{BV} to the crystal $ladd_\ell$. Section 
\ref{gen_crystal} transfers the crystal theorems on $ladd_\ell$ to theorems on $reg_\ell$ via the isomorphism
described in \cite{B}.

\subsection{Background and Previous Results}
Let $\lambda$ be a partition of $n$ (written $\lambda \vdash n$) and $\ell \geq 3$ be an integer. We will use the convention $(x,y)$ to denote the box which sits in the $x^{\textrm{th}}$ row and the $y^{\textrm{th}}$ column of the Young diagram of $\lambda$. We denote the transpose of $\lambda$ by $\lambda'$. Sometimes the shorthand $(a^k)$ will be used to represent the rectangular partition which has $k$-parts, all of size $a$. $\mathcal{P}$ will denote the set of all partitions. 
An \textbf{$\ell$-regular partition} is one in which no part occurs $\ell$ or more times. The \textbf{length of a partition} $\lambda$ will be the number of nonzero parts of $\lambda$ and will be denoted $len(\lambda)$. If $(x,y)$ is a box in the Young diagram of $\lambda$, the \textbf{residue} of $(x,y)$ is $y-x\mod \ell$.

The \textbf{hook length} of the $(a,c)$ box of $\lambda$ is defined to be the number of boxes to the right of  or below the box $(a,c)$, including the box $(a,c)$ itself. It will be denoted \textbf{$h_{(a,c)}^{\lambda}$}.

An \textbf{$\ell$-rim hook} in $\lambda$ is a connected set of $\ell$ boxes in the Young diagram of $\lambda$, containing no $2 \times 2$ square, such that when it is removed from $\lambda$, the remaining diagram is the Young diagram of some other partition.

Any partition which has no $\ell$-rim hooks is called an \textbf{$\ell$-core}. 
Equivalently, $\lambda$ is an $\ell$-core if for every box $(i,j) \in \lambda$, $\ell \nmid h_{(i,j)}^\lambda$. Any partition $\lambda$ has an 
\textbf{$\ell$-core}, which is obtained by removing $\ell$-rim hooks  from the outer edge while at 
each step the removal of a hook is still a (non-skew) partition. The core is uniquely determined from 
the partition, independently of choice of  successively removing rim hooks. See \cite{JK} for more details.

$\ell$-rim hooks which are horizontal (whose boxes are contained in one row of a partition) will be called \textbf{horizontal $\ell$-rim hooks}. $\ell$-rim hooks which are not will be called \textbf{non-horizontal $\ell$-rim hooks}.
An $\ell$-rim hook contained entirely in a single column of the Young diagram of a partition will be called a \textbf{vertical $\ell$-rim hook}. $\ell$-rim hooks not contained in a single column will be called \textbf{non-vertical $\ell$-rim hooks}.
Two connected sets of boxes will be called \textbf{adjacent} if there exist boxes in each which share an edge.

\begin{example} Let $\lambda = (3,2,1)$ and let $\ell=3$. Then the boxes $(1,2), (1,3)$ and $(2,2)$ comprise a (non-vertical, non-horizontal) 3-rim hook. After removal of this $3$-rim hook, the remaining partition is $(1,1,1)$, which is a vertical 3-rim hook. Hence the 3-core of $\lambda$ is the empty partition. These two 3-rim hooks are adjacent.

\begin{center}
$\tableau{ \mbox{} &\mbox{}&\mbox{}\\
\mbox{}&\mbox{}\\
\mbox{}}$
\end{center}
\end{example}

\begin{example} Let $\lambda = (4,1,1,1)$ and $\ell=3$. Then $\lambda$ has two $3$-rim hooks (one horizontal and one vertical). They are not adjacent.

\begin{center}
$\tableau{\mbox{}&\mbox{}&\mbox{}&\mbox{}\\
\mbox{}\\
\mbox{}\\
\mbox{}}$
\end{center}
\end{example}

\begin{definition}\label{lpartition}
An \textbf{$\ell$-partition} is an $\ell$-regular partition containing no removable non-horizontal $\ell$-rim hooks, such that after removing any number of horizontal $\ell$-rim hooks, the remaining diagram still has no removable non-horizontal $\ell$-rim hooks.
\end{definition}

We will study combinatorics related to the finite Hecke algebra $H_n(q)$. For a definition of this algebra, see for instance \cite{BV}.
In this paper we will always assume that $q \in \mathbb{F}$ is a primitive $\ell^{th}$ root of unity in a field $\mathbb{F}$ of characteristic zero.

Similar to the symmetric group, a construction of the Specht module $S^{\lambda} = S^{\lambda}[q]$ exists for $H_n(q)$ (see \cite{DJ}).
Let $\ell$ be an integer greater than $1$. Let 

\[m_\ell(k) = \left\{ 	\begin{array}{ll}
			  1 &  \textrm{ $\ell \mid k$}\\
			  0 &  \textrm{ $\ell \nmid k$}. 
			\end{array} \right.\]
It is known that over the finite Hecke algebra $H_n(q)$, when $q$ is a primitive $\ell^{th}$ root of unity, the Specht module $S^{\lambda}$ for an $\ell$-regular partition $\lambda$  is irreducible  if and only if 
\[\begin{array}{lccrr} (\star) &  m_\ell(h_{(a,c)}^\lambda) = m_\ell(h_{(b,c)}^\lambda) & \textrm{for all pairs $(a,c)$, $(b,c) \in \lambda$} \end{array}\]
(see \cite{JM}).   
In \cite{BV}, we proved the following.
\begin{theorem}\label{main_theorem_l_partitions}
A partition is an $\ell$-partition if and only if it is $\ell$-regular and satisfies ($\star$).
\end{theorem}

Work of Lyle \cite{L} and Fayers \cite{F} settled the following conjecture of James and Mathas.
\begin{theorem}\label{JM_irred} Suppose $\ell > 2$. Let $\lambda$ be a partition. Then $S^{\lambda}$ is reducible if and only if there exist boxes (a,b) (a,y) and (x,b) in the Young diagram of $\lambda$ for which:
\begin{itemize}
\item $m_\ell(h_{(a,b)}^\lambda) =1$,
\item $m_\ell(h_{(a,y)}^\lambda) = m_\ell(h_{(x,b)}^\lambda)= 0$ .
\end{itemize}
\end{theorem}

A partition which has no such boxes is called an \textbf{$(\ell,0)$-JM partition}. 
Equivalently, $\lambda$ is an $(\ell,0)$-JM partition if and only if the Specht module 
$S^\lambda$ is irreducible.

\subsubsection{Ladders}
Let $\lambda$ be a partition and let $\ell>2$ be a fixed integer. For any box $(a,b)$ in the Young diagram of 
$\lambda$, the \textbf{ladder} of $(a,b)$ is the set of all positions $(c,d)$ (here $c,d \geq 1$ are integers) which satisfy $\frac{c-a}{d-b} =
 \ell-1$.

\begin{remark}
The definition implies that two boxes in the same ladder will share the same residue. An $i$-ladder will be a
 ladder which has residue $i$.
\end{remark}

\subsubsection{Regularization} Regularization is a map which takes a partition to a $p$-regular partition.  For a given $\lambda$, move all of the boxes up to the top of their respective ladders. 
The result is a partition, and that partition is called the \textbf{regularization} of $\lambda$, and is 
denoted $\mathcal{R} \lambda$. The following theorem contains facts about regularization originally 
due to James \cite{J} (see also \cite{JM}).
\begin{theorem}\label{reg_prop} Let $\lambda$ be a partition. Then
\begin{itemize}
\item $\mathcal{R} \lambda$ is $\ell$-regular
\item $\mathcal{R} \lambda = \lambda$ if and only if $\lambda$ is $\ell$-regular.
\end{itemize}
\end{theorem}

Regularization provides us with an equivalence relation on the set of partitions. Specifically, we say $\lambda \sim \mu$ if $\mathcal{R} \lambda = \mathcal{R} \mu$. The equivalence classes are called \textbf{regularization classes}, and the class of a partition $\lambda$ is denoted $\mathcal{RC}(\lambda) := \{ \mu \in \mathcal{P} : \mathcal{R}\mu = \mathcal{R}\lambda \}$.

All of the irreducible representations of $H_n(q)$ have been constructed when $q$ is a primitive $\ell^{th}$ root of unity. These modules are indexed by $\ell$-regular partitions $\lambda$, and are called $D^{\lambda}$. $D^{\lambda}$ is the unique simple quotient of $S^{\lambda}$ (see \cite{DJ} for more details). In particular $D^{\lambda} = S^{\lambda}$ if and only if $S^{\lambda}$ is irreducible and $\lambda$ is $\ell$-regular. For $\lambda$ not necessarily $\ell$-regular, $S^{\lambda}$ is irreducible if and only if there exists an $\ell$-regular partition $\mu$ so that $S^{\lambda} \cong D^{\mu}$. An $\ell$-regular partition $\mu$ for which $S^\lambda = D^\mu$ for some $\lambda$ will be called a \textbf{weak $\ell$-partition}.

\begin{theorem}\label{deco}[James \cite{J}, \cite{J2}] Let $\lambda$ be any partition. Then
the irreducible representation $D^{\mathcal{R}\lambda}$ occurs as a multiplicity one composition factor of 
$S^\lambda$. In particular, if $\lambda$ is an $(\ell,0)$-JM partition, then $S^\lambda = D^{\mathcal{R}\lambda}$.
\end{theorem}

\section{Classifying $(\ell,0)$-JM partitions by their Removable $\ell$-Rim Hooks}\label{removing}
\subsection{Motivation}
In this section we give a new description of $(\ell,0)$-JM partitions. This condition is related to how $\ell$-rim hooks are removed from a partition and is a generalization of Theorem 2.1.6 in \cite{BV} about $\ell$-partitions. The condition we give will be used in several proofs throughout this paper. 
\subsection{Removing $\ell$-Rim Hooks and $(\ell,0)$-JM partitions}
\begin{definition}\label{generalize_l_partition} Let $\lambda$ be a partition. Let $\ell > 2$. Then $\lambda$ is a \textbf{generalized $\ell$-partition} if:
\begin{enumerate} 
\item $\lambda$ has only horizontal and vertical $\ell$-rim hooks;
\item\label{gen_condition} for any vertical (resp. horizontal) $\ell$-rim hook $R$ of $\lambda$ and any horizontal (resp. vertical) $\ell$-rim hook $S$ of $\lambda \setminus R$, $R$ and $S$ are not adjacent;
\item after removing any set of horizontal and vertical $\ell$-rim hooks from the Young diagram of $\lambda$, the remaining partition satisfies (1) and (2).
\end{enumerate}
\end{definition}

\begin{example} \label{6111}
Let $\lambda = (3,1,1,1)$. $\lambda$ has a vertical $3$-rim hook $R$ containing the boxes $(2,1), (3,1), (4,1)$.  Removing R leaves a horizontal $3$-rim hook $S$ containing the boxes $(1,1), (1,2), (1,3)$. $S$ is adjacent to $R$, so $\lambda$ is not a generalized $3$-partition.

\begin{center}
$\tableau{ S & S& S \\ R \\ R \\ R}$
\end{center}
\end{example}

\begin{remark}
We will sometimes abuse notation and say that $R$ and $S$ in Example \ref{6111} are adjacent vertical and horizontal $\ell$-rim hooks. The meaning here is not that they are both $\ell$-rim hooks of $\lambda$ ($S$ is not an $\ell$-rim hook of $\lambda$), but rather that they are an example of a violation of 
condition \ref{gen_condition} from Definition \ref{generalize_l_partition}.
\end{remark}

We will need a few lemmas before we come to our main theorem of this section, which states that the notions of $(\ell,0)$-JM partitions and generalized $\ell$-partitions are equivalent.
The next lemma simplifies the condition for being an $(\ell,0)$-JM partition and is used in the proof of Theorem \ref{main_theorem_JM}.

\begin{lemma}\label{rearrange}
Suppose $\lambda$ is not an $(\ell,0)$-JM partition. Then there exist boxes $(c,d)$, $(c,w)$ and $(z, d)$ with $c < z$, $d < w$, and
 $\ell \mid h_{(c,d)}^{\lambda}$, $\ell \nmid h_{(c,w)}^{\lambda}, h_{(z,d)}^{\lambda}$.
 \end{lemma}

\begin{proof}
By assumption there exist boxes $(a,b)$, $(a,y)$ and $(x, b)$ where $\ell \mid h_{(a,b)}^{\lambda}$ and $\ell \nmid h_{(a,y)}^{\lambda}, h_{(x,b)}^{\lambda}$. If $a<x$ and $b<y$ then we are done. The other cases follow below:

Case 1: $x<a$ and $y<b$. Assume no triple exists satisfying the statement of the lemma. Then either all boxes to the right of the $(a,b)$ box will have hook lengths divisible by $\ell$, or all boxes below will. Without loss of generality, suppose that all boxes below the $(a,b)$ box have hook lengths divisible by $\ell$. Let $c<a$ be the largest integer so that $\ell \nmid h_{(c,b)}$. Let $z = c+1$. Then one of the boxes $(c,b+1), (c,b+2), \dots (c,b+\ell-1)$ has a hook length divisible by $\ell$. This is because the box $(h,b)$ at the bottom of column $b$ has a hook length divisible by $\ell$, so the hook lengths $h_{(c,b)}^{\lambda} = h_{(c,b+1)}^{\lambda} +1 = \dots = h_{(c, b+\ell-1)}^{\lambda} + \ell -1$. Suppose it is $(c,d)$. Then $\ell \nmid h_{(z,d)}^{\lambda}$ since $h_{(z,b)} = h_{(z,d)} + d-b$ and $d-b < \ell$.

 If $d \neq b+ \ell-1$ or $h_{(h,b)}^{\lambda} > \ell$ then letting $w = d+1$ gives $(c, w)$ to the right of $(c,d)$ so that $\ell \nmid h_{(c,w)}^{\lambda}$ (in fact $h_{(c,w)}^{\lambda} = h_{(c,d)}^{\lambda} - 1$).

If $d = b+\ell-1$ and $h_{(h,b)}^{\lambda} = \ell$ then there is a box in position $(c, d+1)$ with hook length
 $h_{(c,d+1)}^{\lambda} = h_{(c,d)}^{\lambda} - 2$ since there must be a box in the position $(h-1, d+1)$, 
due to the fact that $\ell \mid h_{(h-1, b)}^\lambda$ and $h_{(h-1, b)}^\lambda > \ell$ if $h-1 \neq c$ and $h_{(h-1,d)}^\lambda >\ell$ if $h-1=c$. Letting $w = d+1$ 
again yields $\ell \nmid h_{(c,w)}^\lambda$. Note that this requires that $\ell > 2$. In fact if $\ell =2$ 
we cannot even be sure that there is a box in position $(c, d+1)$. 

Case 2: $x<a$ and $y>b$.
If there was a box $(n,b)$ ($n > a$) with a hook length not divisible by $\ell$ then we would be done. So we can assume that all hook lengths in column $b$ below row $a$ are divisible by $\ell$. Let $c<a$ be the largest integer so that $\ell \nmid h_{(c,b)}^{\lambda}$. Let $z = c+1$. Similar to Case 1 above, we find a $d$ so that $\ell \mid h_{(c,d)}^\lambda$. Then $\ell \nmid h_{(z,d)}^\lambda$ and by the same argument as in Case 1, if we let $w = d+1$ then $\ell \nmid h_{(c,w)}^\lambda$. 

Case 3: $x>a$ and $y<b$. Then apply Case 2 to $\lambda'$.
\end{proof}

\begin{lemma}\label{adding}
Suppose $\lambda$ is not an $(\ell,0)$-JM partition. Then a partition obtained from $\lambda$ by adding a horizontal or
 vertical $\ell$-rim hook is also not an $(\ell,0)$-JM partition.
\end{lemma}

\begin{proof}
Let us suppose that we are adding a horizontal $\ell$-rim hook $R$ to a row $r$ in $\lambda$ to produce a partition $\mu$. By Lemma \ref{rearrange}, we can assume that there are boxes $(c,d)$, $(c,w)$ and $(z,d)$ as stated in the lemma. The only complication arises when $R$ is directly below one or more of these boxes. When this is the case, the fact that $R$ is completely horizontal implies that adjacent boxes also below $R$ will have hook lengths which differ by exactly one. This allows us to find new boxes $(c,d)$, $(c,w)$ and $(z,d)$ which satisfy Lemma \ref{rearrange}. Therefore $\mu$ is also not an $(\ell,0)$-JM partition.
\end{proof}

\begin{theorem}\label{main_theorem_JM} A partition is an $(\ell,0)$-JM partition if and only if it is a generalized $\ell$-partition.
\end{theorem}
\begin{proof}
Suppose that $\lambda$ is not a generalized $\ell$-partition.
 Then remove non-adjacent horizontal and vertical $\ell$-rim hooks until you obtain a partition 
$\mu$ which has either a non-vertical non-horizontal $\ell$-rim hook, or adjacent horizontal and 
vertical $\ell$-rim hooks. If there is a non-horizontal, non-vertical $\ell$-rim hook in $\mu$, 
let's say the $\ell$-rim hook has southwest most box $(a,b)$ and northeast most box $(c,d)$. 
Then $\ell \mid h_{(c,b)}^{\mu}$ but  
$\ell \nmid h_{(a,b)}^{\mu} , h_{(c,d)}^{\mu}$ since $h_{(a,b)}^\mu, h_{(c,d)}^\mu < \ell$. Therefore, $\mu$ is not an $(\ell,0)$-JM partition. By Lemma 
\ref{adding}, $\lambda$ is not an $(\ell,0)$-JM partition. Similarly, if $\mu$ has adjacent vertical
 and horizontal $\ell$-rim hooks, then let $(a,b)$ be the southwest most box in the vertical $\ell$-rim
 hook and let $(c,d)$ be the position of the northeast most box in the horizontal $\ell$-rim
 hook (we may assume that the horizontal rim hook is to the north east of the vertical one, otherwise the pair would also form a non-vertical, non-horizontal $\ell$-rim hook). Again, $\ell \mid h_{(c,b)}^{\mu}$ but  $\ell \nmid h_{(a,b)}^{\mu} , h_{(c,d)}^{\mu}$. 
Therefore $\mu$ cannot be an $(\ell,0)$-JM partition, so $\lambda$ is not an $(\ell,0)$-JM partition.

Conversely, let $n$ be the smallest integer such that there exists a partition $\lambda \vdash n$ which is not 
an $(\ell,0)$-JM partition but is a generalized
 $\ell$-partition. Then by Lemma \ref{rearrange} there are boxes $(a,b)$, $(a,y)$ and $(x,b)$ 
with $a < x$ and $b < y$, which satisfy $\ell \mid h_{(a,b)}^{\lambda}$, and 
$\ell \nmid h_{(a,y)}^\lambda, h_{(x,b)}^\lambda$ . Form a new partition $\mu$ by taking all of the 
boxes $(m,n)$ in $\lambda$ such that $m \geq a$ and $n \geq b$. Since $\lambda$ was a generalized $\ell$-partition, 
$\mu$ must be also. If $\mu \neq \lambda$ then we have found a partition $\mu \vdash m$ for $m <n$, which is a contradiction.
So we may assume that $a, b =1$. 

From the definition of $\ell$-cores, we know that there must exist a removable 
$\ell$-rim hook from $\lambda$, since $\ell \mid h_{(1,1)}^{\lambda}$. Since $\lambda$ is a generalized 
$\ell$-partition, the $\ell$-rim hook must be either horizontal or vertical.
 Without loss of generality, suppose we have a horizontal $\ell$-rim 
hook which can be removed from $\lambda$. Let the resulting partition be denoted $\nu$. 

If $h_{(1,1)}^\nu = h_{(1,1)}^\lambda - 1$, then the horizontal $\ell$-rim hook was removed from the last
 row of $\lambda$,
which was of length exactly $\ell$. If this is the case then $h_{(1,\ell)}^\lambda \equiv 1 \mod \ell$, 
$h_{(1,1)}^\lambda \equiv 0 \mod \ell$
and $h_{(x,\ell)}^\lambda \equiv h_{(x,1)}^\lambda +1 \mod \ell$. Hence $\ell \mid h_{(1,\ell)}^\nu$ 
(since $h_{(1,\ell)}^\nu = h_{(1,\ell)}^\lambda - 1$),
$\ell \nmid h_{(1,1)}^\nu$, $\ell \nmid h_{(x,\ell)}^\nu$. Therefore $\nu$ is not an $(\ell,0)$-JM partition,
 but it is 
a generalized $\ell$-partition. The existence of such a partition is a contradiction. So we know that removing a
horizontal $\ell$-rim hook from $\lambda$ cannot change the value of $h_{(1,1)}^\lambda$ by $1$. This is also true
for vertical $\ell$-rim hooks.

Now we may assume that removing horizontal or vertical 
$\ell$-rim hooks from $\lambda$ will not change that $\ell$ divides the hook length in the $(1,1)$ position (because removing 
each $\ell$-rim hook will change the hook length $h_{(1,1)}^\lambda$ by either $0$ or $\ell$). Therefore we 
can keep removing $\ell$-rim hooks until we have have removed box (1,1) entirely, in which case 
the remaining partition had a horizontal $\ell$-rim hook adjacent to a vertical $\ell$-rim hook 
(since both $(x,b)$ and $(a,y)$ must have been removed, the $\ell$-rim hooks could not have been 
exclusively horizontal or vertical). This contradicts $\mu$ being a generalized $\ell$-partition.
\end{proof}

\begin{example} Let $\lambda = (10,8,3,2^2, 1^5)$. Then $\lambda$ is a generalized 3-partition and a $(3,0)$-JM partition. $\lambda$ is drawn below with each hook length $h_{(a,b)}^\lambda$ written in the box $(a,b)$ and the possible removable $\ell$-rim hooks outlined. Also, hook lengths which are divisible by $\ell$ are underlined.

\begin{center}
$\tableau{ 19&13&10&8&7&\underline{6}&5&4&2&1\\
16&10&7&5&4&\underline{3}&2&1\\
10&4&1\\
8&2\\
7&1\\
5\\
4\\
\underline{3}\\
2\\
1
}
\linethickness{2.5pt}
\put (-180,-162){\line(0,1){54}}
\put (-162,-162){\line(0,1){54}}
\put (-180,-162){\line(1,0){18}}
\put (-180,-108){\line(1,0){18}}
\put (-54,0){\line(1,0){54}}
\put (-54,18){\line(1,0){54}}
\put (-54,0){\line(0,1){18}}
\put (0,0){\line(0,1){18}}
\put (-90,-18){\line(1,0){54}}
\put (-90,0){\line(1,0){54}}
\put (-90,-18){\line(0,1){18}}
\put (-36,-18){\line(0,1){18}}
$
\end{center}
\end{example}

\begin{lemma}\label{JMAAR}
An $(\ell,0)$-JM partition $\lambda$ cannot have a removable and two addable partitions of the same residue.
\end{lemma}

\begin{proof}

 Label the removable box $n_1$.  Label the addable boxes $n_2$ and $n_3$ 
(without loss of generality, $n_2$ is in a row above $n_3$). There are three cases to consider.

The first case is that $n_1$ is above $n_2$ and $n_3$. Then the hook length in the row of $n_1$
 and column of $n_3$ is divisible by $\ell$, but the hook length in the row of $n_2$ and column
 of $n_3$ is not. Also, the hook length for box $n_1$ is 1, which is not divisible by $\ell$. 

The second case is that $n_1$ is in a row between the row of $n_2$ and $n_3$. In this case, $\ell$
 divides the hook length in the row of $n_1$ and column of $n_3$. Also $\ell$ does not divide the
 hook length in the row of $n_2$ and column of $n_3$, and the hook length for the box $n_1$ is 1. 

The last case is that $n_1$ is below $n_2$ and $n_3$. In this case, $\ell$ divides the hook length
 in the column of $n_1$ and row of $n_2$, but $\ell$ does not divide the hook length in the column
 of $n_3$ and row of $n_2$. Also the hook length for the box $n_1$ is 1.

\end{proof}

\section{Decomposition of $(\ell,0)$-JM Partitions}\label{construct_JM} 
\subsection{Motivation} In \cite{BV} we gave a decomposition of $\ell$-partitions. In this section 
we give a similar decomposition for all $(\ell,0)$-JM partitions. This decomposition is important for 
the proofs of the theorems in later sections.

\subsection{Decomposing $(\ell,0)$-JM partitions}
Let $\mu$ be an $\ell$-core with $\mu_1-\mu_2 < \ell-1$ and $\mu_1'-\mu_2' < \ell-1$. Let $r,s \geq 0$. Let $\rho$ and $\sigma$ be partitions with $len(\rho) \leq r+1$ and $len(\sigma) \leq s+1$. If $\mu = \emptyset$ then we require at least one of $\rho_{r+1}, \sigma_{s+1}$ to be zero. Following the construction of \cite{BV}, we construct a partition corresponding to $(\mu,r,s,\rho,\sigma)$ as follows. Starting with $\mu$, attach $r$ rows above $\mu$, with each row $\ell-1$ boxes longer than the previous. Then attach $s$ columns to the left of $\mu$, with each column $\ell-1$ boxes longer than the previous. This partition will be denoted $(\mu,r,s,\emptyset,\emptyset)$. Formally, if $\mu = (\mu_1,\mu_2,\dots,\mu_m)$ then $(\mu,r,s,\emptyset,\emptyset)$ represents the partition (which is an $\ell$-core):  
\begin{center} $(s+\mu_1+r(\ell-1), s+\mu_1+ (r-1)(\ell-1), \dots, s+\mu_1+\ell-1, s+\mu_1, $

$s+\mu_2, \dots, 
s+\mu_m, s^{\ell-1}, (s-1)^{\ell-1}, \dots, 1^{\ell-1})$
\end{center}

where $s^{\ell-1}$ stands for $\ell-1$ copies of $s$.
Now to the first $r+1$ rows attach $\rho_i$ horizontal $\ell$-rim hooks to row $i$. Similarly, to the first $s+1$ columns, attach $\sigma_j$ vertical $\ell$-rim hooks to column $j$. The resulting partition $\lambda$ corresponding to $(\mu,r,s,\rho,\sigma)$ will be
\begin{center}

$\lambda = (s+\mu_1+r(\ell-1)+ \rho_1 \ell, s+\mu_1+(r-1)(\ell-1)+ \rho_2 \ell, \dots, $

$s+\mu_1+ (\ell-1) + \rho_r 
\ell,  s+\mu_1+\rho_{r+1}\ell, s+\mu_2, s+\mu_3, \dots, $

$s+\mu_m,
(s+1)^{\sigma_{s+1}\ell}, s^{\ell-1+(\sigma_s - \sigma_{s+1})\ell}, (s-1)^{\ell-1 + (\sigma_{s-1} - \sigma_s)\ell}, \dots , 1^{\ell-1 + (\sigma_1-\sigma_2)\ell}).$
\end{center}

We denote this decomposition as $\lambda \approx (\mu,r,s,\rho,\sigma)$.

\begin{example}

Let $\ell = 3$ and $(\mu,r,s,\rho,\sigma) = ((1),3,2,(2,1,1,1),(2,1,0))$. Then $((1),3,2, \emptyset, \emptyset)$ is drawn below, with $\mu$ framed.

\vspace{.3in}

$\begin{array}{cc}

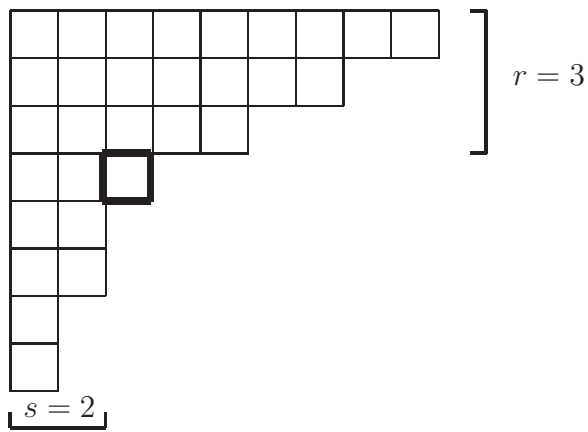
\linethickness{2.5pt}
\put (45,-36){\line(1,0){18}}
\put (45,-54){\line(1,0){18}}
\put (45,-54){\line(0,1){18}}
\put (63,-54){\line(0,1){18}}
\linethickness{1pt}
\put (10,-140){\line(1,0){36}}
\put (10,-140){\line(0,1){6}}
\put (46,-140){\line(0,1){6}}
\put (15,-136){$s = 2$}
\put (190,-36){\line(0,1){54}}
\put (190,-36){\line(-1,0){6}}
\put (190,18){\line(-1,0){6}}
\put (200,-10){$r = 3$}
&
\tableau{\mbox{}&\mbox{}&\mbox{}&\mbox{}&\mbox{}&\mbox{}&\mbox{}&\mbox{}&\mbox{}\\
\mbox{}&\mbox{}&\mbox{}&\mbox{}&\mbox{}&\mbox{}&\mbox{}\\
\mbox{}&\mbox{}&\mbox{}&\mbox{}&\mbox{}\\
\mbox{}&\mbox{}&\mbox{}\\
\mbox{}&\mbox{}\\
\mbox{}&\mbox{}\\
\mbox{}\\
\mbox{}
}
\end{array}$

\vspace{.3in}

$((1),3,2,(2,1,1,1),(2,1,0))$ is drawn below, now with $((1),3,2, \emptyset, \emptyset)$ framed.

\vspace{.3in}

$\begin{array}{cc}
\linethickness{2.5pt}
\put (63,-36){\line(1,0){36}}
\put (45,-54){\line(1,0){18}}
\put (45,-90){\line(0,1){36}}
\put (63,-54){\line(0,1){18}}
\put (27,-90){\line(1,0){18}}
\put (27,-126){\line(0,1){36}}
\put (9,-126){\line(1,0){18}}
\put (9,-126){\line(0,1){144}}
\put (9,18){\line(1,0){162}}
\put (171,0){\line(0,1){18}}
\put (135,0){\line(1,0){36}}
\put (135,-18){\line(0,1){18}}
\put (99,-18){\line(1,0){36}}
\put (99,-36){\line(0,1){18}}

&
\tableau{\mbox{}&\mbox{}&\mbox{}&\mbox{}&\mbox{}&\mbox{}&\mbox{}&\mbox{}&\mbox{}&\mbox{}&\mbox{}&\mbox{}&\mbox{}&\mbox{}&\mbox{}\\
\mbox{}&\mbox{}&\mbox{}&\mbox{}&\mbox{}&\mbox{}&\mbox{}&\mbox{}&\mbox{}&\mbox{}\\
\mbox{}&\mbox{}&\mbox{}&\mbox{}&\mbox{}&\mbox{}&\mbox{}&\mbox{}\\
\mbox{}&\mbox{}&\mbox{}&\mbox{}&\mbox{}&\mbox{}\\
\mbox{}&\mbox{}\\
\mbox{}&\mbox{}\\
\mbox{}&\mbox{}\\
\mbox{}&\mbox{}\\
\mbox{}&\mbox{}\\
\mbox{}\\
\mbox{}\\
\mbox{}\\
\mbox{}\\
\mbox{}
}
\end{array}
$

\end{example}

\begin{theorem}\label{construct_JMs}
If $\lambda \approx (\mu,r,s,\rho,\sigma)$ (with at least one of $\rho_{r+1}, \sigma_{s+1} = 0$ if $\mu=\emptyset$), then $\lambda$ is an $(\ell,0)$-JM partition. Conversely, all $(\ell,0)$-JM partitions are of this form.
\end{theorem}
\begin{proof}
First, note that $(\mu,r,s,\emptyset,\emptyset)$ is an $\ell$-core. This can be seen as no $\ell$-rim hooks can be removed from $\mu$, since $\mu$ is an $\ell$-core, so any $\ell$-rim hooks which can be removed from $(\mu,r,s,\emptyset,\emptyset)$ must contain at least one box in either the first $r$ rows or $s$ columns. But it is clear that no $\ell$-rim hook can go through one of these rows or columns.

If $\lambda \approx (\mu,r,s,\rho,\sigma)$ then it is clear by construction that $\lambda$ satisfies the criterion for a generalized $\ell$-partition (see Definition \ref{generalize_l_partition}). By Theorem \ref{main_theorem_JM}, $\lambda$ is an $(\ell,0)$-JM partition.

Conversely, if $\lambda$ is an $(\ell,0)$-JM partition then by Theorem \ref{main_theorem_JM} its only removable $\ell$-rim hooks are horizontal or vertical. Let $\rho_i$ be the number of removable horizontal $\ell$-rim hooks in row $i$ which are removed in going to the $\ell$-core of $\lambda$, and let $\sigma_j$ be the number of removable vertical $\ell$-rim hooks in column $j$ (since $\lambda$ has no adjacent $\ell$-rim hooks, these numbers are well defined). Once all $\ell$-rim hooks are removed, let $r$ (resp. $s$) be the number of rows (resp. columns) whose successive differences are $\ell-1$. It is then clear that $len(\rho) \leq r+1$, since if it wasn't then the two rows $r+1$ and $r+2$ would combine to form a non-vertical, non-horizontal $\ell$-rim hook. Similarly, $len(\sigma) \leq s+1$.
 Removing these topmost $r$ rows and leftmost $s$ columns leaves an $\ell$-core $\mu$. 
Then $\lambda \approx (\mu, r, s, \rho, \sigma)$. If $\mu =  \emptyset$ and $\rho_{r+1}, \sigma_{s+1}>0$ then $\lambda$
would have (after removal of horizontal and vertical $\ell$-rim hooks) a horizontal $\ell$-rim hook adjacent to a vertical
$\ell$-rim hook.
\end{proof}

Further in the text, we will make use of Theorem \ref{construct_JMs}. Many times we will show that a partition $\lambda$ is an $(\ell,0)$-JM partition by giving an explicit decomposition of $\lambda$ into $(\mu,r,s,\rho,\sigma)$.

\begin{remark}
 This decomposition can be used to count the number of $(\ell,0)$-JM partitions in a given block. For more details, see the author's Ph.D. thesis \cite{Bphd}.
\end{remark}

\section{Extending Theorems to the Crystal $ladd_\ell$}\label{extend}

In \cite{MM}, Misra and Miwa built a model (denoted here as $reg_\ell$) of the basic representation $B(\Lambda_0)$ of $\widehat{\mathfrak{sl}_\ell}$ using $\ell$-regular partitions as nodes of the graph. Their crystal operators $\widetilde{e_i}$ (resp. $\widetilde{f_i}$) are maps which remove (resp. add) a box to a partition.

In \cite{B}, I built a crystal model (denoted here as $ladd_\ell$) of $B(\Lambda_0)$ which had a certain type of partitions as nodes of the graph. The crystal operators of my model, named $\widehat{e_i}$ and $\widehat{f_i}$, removed and added boxes in a similar manner. I showed that my model was the basic crystal $B(\Lambda_0)$ by showing that the map $\mathcal{R}$ described above actually gave one direction of the crystal isomorhism (taking a partition in my model and making it $\ell$-regular). 

To be more specific, to a partition $\lambda$, and a residue $i \in \{0, \dots, \ell-1 \}$, we put a $-$ in every box of $\lambda$ which is removable and has residue $i$. We also put a $+$ in every position adjacent to $\lambda$ which is addable and has residue $i$. We make a word out of these $-$'s and $+$'s. In the Misra Miwa model, the word is read from the bottom of the partition to the top. In the ladder crystal model, the word is read from leftmost ladder to rightmost ladder, reading each ladder from top to bottom. The reduced word is then obtained by successive cancelation of adjacent pairs $-$ $+$. We can now define $\widetilde{e_i} \lambda$ (resp. $\widehat{e_i} \lambda$) as the partition obtained by removing from $\lambda$ the box corresponding to the leftmost $-$ in the reduced word of the Misra Miwa ordering (resp. ladder ordering). Similarly, $\widetilde{f_i } \lambda$ (resp. $\widehat{f_i} \lambda$) is the partition obtained by adding a box to $\lambda$ corresponding to the rightmost $+$ in the reduced word of the Misra Miwa ordering (resp. ladder ordering). To see these rules in more detail, with examples, see \cite{B}.

Through the rest of this paper, 

\[\varepsilon = \varepsilon_i (\lambda) = max \{n : \widetilde{e_i}^n \lambda \neq 0 \},\] \[\varphi = \varphi_i (\lambda) = max \{n : \widetilde{f_i}^n \lambda \neq 0 \},\] \[\widehat{\varepsilon} = \widehat{\varepsilon}_i (\lambda) = max \{n : \widehat{e_i}^n \lambda \neq 0\},\] \[\widehat{\varphi} = \widehat{\varphi}_i (\lambda)= max \{n : \widehat{f_i}^n \lambda \neq 0\}.\]

\subsection{Previous results on the crystal $reg_\ell$}
In \cite{BV} we proved the following theorem about $\ell$-partitions in the crystal $reg_\ell$. 
\begin{theorem}\label{top_and_bottom}
Suppose that $\lambda$ is an $\ell$-partition and $0\leq i < \ell$. Then
\begin{enumerate}
\item\label{f} $\widetilde{f}_{i}^{\varphi} \lambda$ is an $\ell$-partition,
\item\label{e} $\widetilde{e}_{i}^{\varepsilon} \lambda$ is an $\ell$-partition.

\item\label{f_theorem} $\widetilde{f}_{i}^{k} \lambda$ is not an $\ell$-partition for $0<k < \varphi -1,$
\item\label{e_theorem} $\widetilde{e}_{i}^k \lambda$ is not an $\ell$-partition for $1<k< \varepsilon$.
\end{enumerate}
\end{theorem}

In this paper, we generalize the above Theorem \ref{top_and_bottom} to weak $\ell$-partitions. We first give the statement of our new theorem.

\begin{theorem}\label{top_and_bottom_weak}
Suppose that $\lambda$ is a weak $\ell$-partition and $0\leq i < \ell$. Then
\begin{enumerate}
\item $\widetilde{f}_{i}^{\varphi} \lambda$ is a weak $\ell$-partition,
\item $\widetilde{e}_{i}^{\varepsilon} \lambda$ is a weak $\ell$-partition.
\item $\widetilde{f}_{i}^{k} \lambda$ is not a weak $\ell$-partition for $0<k < \varphi -1,$
\item $\widetilde{e}_{i}^k \lambda$ is not a weak $\ell$-partition for $1<k< \varepsilon$.
\end{enumerate}
\end{theorem}

\subsection{Crystal theoretic results for $ladd_\ell$ and $(\ell,0)$-JM partitions}
For a proof of these new theorems, we will start by proving analogous statements in the ladder crystal $ladd_\ell$.
To do this, we will first need some lemmas. 

\begin{lemma}\label{cores}
 All $\ell$-cores are nodes of $ladd_\ell$. In particular, If $\lambda$ is an $\ell$-core, then $\widehat{\varphi} = \varphi$ and $\widehat{f}_i^{\hat{\varphi}} \lambda = \tilde{f}_i^{\varphi} \lambda$.
\end{lemma}

\begin{proof}
$\ell$-cores are unique in their regularization class, so since $\mathcal{R}$ is an isomorphism from $ladd_\ell$ to $reg_\ell$, all $\ell$-cores are nodes of $ladd_\ell$. The second statement is just a consequence of the crystals being isomorphic.
\end{proof}

The following lemma is a well known recharacterization of $\ell$-cores.
\begin{lemma}\label{hook_length_divisible}
A box $x$ has a hook length divisible by $\ell$ if and only if there exists a residue $i$ 
so that the last box in the row of $x$ has residue $i$ and the last box in the column of $x$ 
has residue $i+1$. In particular, any partition which has such a box $x$ is not an $\ell$-core.
\end{lemma}

\begin{lemma}\label{nodes_JM}
If $\lambda$ is an $(\ell,0)$-JM partition then there is 
no ladder in the Young diagram of $\lambda$ which has a $-$ above a $+$.
\end{lemma}

\begin{proof}

Let the coordinates of the $-$ be $(a,b)$, and let the coordinates of the $+$ be $(c,d)$. If $b-d = m$, then the box $(a,d)$ has $h^\lambda_{(a,d)} = m \ell$. Also $h^\lambda_{(a,b)} = 1$. If we can find a box $x$ in column $d$ such that $\ell \nmid h^\lambda_{x}$, then $\lambda$ is not an $(\ell,0)$-JM partition, contradicting the hypothesis. Suppose all the boxes below $(a,d)$ had hook lengths which were multiples of $\ell$. Since $h^\lambda_{(a,d)} = m \ell$, there are $m \ell - m$ total boxes  below $(a,d)$ in $\lambda$. Since the hook lengths of each box decreases down any column, at most $m-1$ of these boxes can have hook lengths divisible by $\ell$ (corresponding to hook lengths $(m-1) \ell, (m-2)\ell, \dots, \ell$). Therefore, the remaining $m\ell - m - (m -1) = m (\ell-2)-1$ must all have hook lengths not divisible by $\ell$. Since $\ell > 2$, $m (\ell-2) - 1 > m-1 \geq 0$, so some box in column $d$ must not be divisible by $\ell$.

\end{proof}

The following lemma will be used in this section for proving our crystal 
theorem generalizations for $(\ell,0)$-JM partitions.

\begin{lemma}\label{cancelation} Let $\lambda$ be an $(\ell,0)$-JM partition. Then the ladder 
$i$-signature of $\lambda$ is the same as the reduced ladder $i$-signature of $\lambda$. In other words, there is no 
$-+$ cancelation in the ladder $i$-signature of $\lambda$. 
\end{lemma}

\begin{proof}
Suppose there is a $-+$ cancelation in the ladder $i$-signature of an $(\ell,0)$-JM partition $\lambda$.
 By Lemma \ref{nodes_JM}, it must be that a removable $i$-box occurs on a ladder to the left of a ladder
 which contains an addable $i$-box. Suppose the removable $i$-box is in position $(a,b)$ and the addable 
$i$-box is in position $(c,d)$. We will suppose that $a > c$ (the case $a<c$ 
is similar). Then $\ell \mid h_{(c,b)}^{\lambda}$. Also $h_{(a,b)}^{\lambda} = 1$ since $(a,b)$ is a 
removable box. Let us suppose that $\ell \mid h_{(c,k)}^\lambda$ for all boxes $(c,k)$ in $\lambda$ to the right of $(c,b)$. The fact that $(a,b)$ is in a ladder to the left of $(c,d)$ is equivalent to the fact that $\frac{d-b}{a-c} >1$, or $d-b > a-c$. By definition $h_{(c,b)}^\lambda =d-b+a-c+1$. The number of positions $(c,k)$ in $\lambda$ for $k \geq b$ can be at most $\frac{h_{(c,b)}}{\ell} = \frac{d-b+a-c+1}{\ell} < \frac{2(d-b) +1}{\ell} < d-b$ since $\ell >2$. In order for $(c,d)$ to be an addable position, we need to have exactly $d-b$ boxes $(c,k)$ for $k \geq b$ in $\lambda$. This contradicition implies that $\lambda$ is not an $(\ell,0)$-JM partition.
\end{proof}

\subsection{Generalizations of the crystal theorems to $ladd_\ell$}
We will now prove an analogue of Theorem \ref{top_and_bottom} for $(\ell,0)$-JM partitions in the ladder crystal $ladd_\ell$.
\begin{theorem}\label{top_and_bottom_JM}
Suppose that $\lambda$ is an $(\ell,0)$-JM partition and $0\leq i < \ell$. Then
\begin{enumerate}
\item\label{f_JM} $\widehat{f}_{i}^{\widehat{\varphi}} \lambda$ is an $(\ell,0)$-JM partition,
\item\label{e_JM} $\widehat{e}_{i}^{\widehat{\varepsilon}} \lambda$ is an $(\ell,0)$-JM partition.
\item\label{f_theorem_JM} $\widehat{f}_{i}^{k} \lambda$ is not an $(\ell,0)$-JM partition for $0<k < \widehat{\varphi} -1,$
\item\label{e_theorem_JM} $\widehat{e}_{i}^k \lambda$ is not an $(\ell,0)$-JM partition for $1<k< \widehat{\varepsilon}$.
\end{enumerate}
\end{theorem}
\begin{proof}
We will prove \eqref{f_JM}; \eqref{e_JM} follows similarly. Suppose that $\lambda \approx 
(\mu,r,s,\rho,\sigma)$ has an addable $m$-box in the first row, and an addable $n$-box in the 
first column for two residues $m,n$. If $m\neq i \neq n$ then $\widehat{f}_i^{\widehat{\varphi}}$
 will only add boxes to the core $\mu$ in the Young diagram of $\lambda$, and not in the first row or column of $\mu$. But 
$\widehat{f}_{i-r+s}^{\widehat{\varphi}} \mu$ will again be a core, by Lemma \ref{cores}. 
Hence $\widehat{f}_i^{\widehat{\varphi}} \lambda \approx (\widehat{f}_{i-r+s}^{\varphi} \mu , r,s,\rho, \sigma)$.

Next we assume $m =  i \neq n$. The partition $\nu \approx (\mu, r, 0, \rho, \emptyset)$ is an $\ell$-partition.
 From Lemma \ref{cancelation}, we have no cancelation of $-+$ in $\lambda$, so that 
$\widehat{f}_{i-s} ^{\widehat{\varphi}_{i-s} (\nu)} \nu = \tilde{f}_{i-s} ^{\varphi_{i-s} (\nu)} \nu$. 
By Theorem \ref{top_and_bottom}, $\widehat{f}_{i-s} ^{\widehat{\varphi}_{i-s} (\nu)} \nu $ is an 
$\ell$-partition. Say 
$\widehat{f}_{i-s} ^{\widehat{\varphi}_{i-s} (\nu)}\nu \approx (\mu', r',0, \rho', \emptyset)$. 
Then $\widehat{f}_i^{\widehat{\varphi}} \lambda \approx (\mu',r',s,\rho',\sigma)$. 
A similar argument works when $m \neq i = n$ by using that the transpose of 
$(\mu, 0, s , \emptyset, \sigma)$ is an $\ell$-partition. 

We now suppose that $m = i = n$. 
In this case, $\lambda$ has an addable $i$-box in the first $r+1$ rows and $s+1$ columns. It may also have
addable $i$-boxes within the core $\mu$. $\lambda$ has no removable $i$-boxes. Thus we get 
$\widehat{f}_i^{\widehat{\varphi}} \lambda \approx (\widehat{f}_{i-r+s}^{\widehat{\varphi}_{i-r+s}(\mu)} \mu, r, s, \rho, \sigma)$ 
is an $(\ell,0)$-JM partition.

 \eqref{f_theorem_JM} follows from \ref{JMAAR}. \eqref{e_theorem_JM} is similar. 

\end{proof}

\subsection{All $(\ell,0)$-JM partitions are nodes of $ladd_\ell$}

\begin{theorem}\label{irreducible_nodes} If $\lambda$ is an $(\ell,0)$-JM partition then $\lambda$ is a node 
of $ladd_\ell$.

\begin{proof}
The proof is by induction on the size of a partition. If the partition has size zero then it is the empty partition which is an $(\ell,0)$-JM partition and is a node of the crystal $ladd_\ell$.

Suppose $\lambda \vdash n>0$ is an $(\ell,0)$-JM partition. Let $i$ be a 
residue such that $\lambda$ has at least one ladder-normal box of residue $i$. We can find such a box since no $-+$ cancellation exists by Lemma \ref{cancelation}. 
Define $\mu$ to be $\widehat{e}_i^{\widehat{\varepsilon}} \lambda$.
 Then $\mu \vdash (n-\hat{\varepsilon})$ is an $(\ell,0)$-JM partition by 
Theorem \ref{top_and_bottom_JM}, of smaller size than $\lambda$. By induction 
$\mu$ is a node of $ladd_\ell$. But $\widehat{f}_i^{\widehat{\varepsilon}} 
\mu= \widehat{f}_i^{\widehat{\varepsilon}} \widehat{e}_i^{\widehat{\varepsilon}}
 \lambda = \lambda$, so $\lambda$ is a node of $ladd_\ell$.
\end{proof}
\end{theorem}

\section{Generalizing Crystal Theorems}\label{gen_crystal}

We can now prove our generalization of Theorem \ref{top_and_bottom}.

\begin{proof}[Proof of Theorem \ref{top_and_bottom_weak}] Let $\lambda$ be a weak $\ell$-partition. Then $D^{\lambda} = S^{\nu}$ for some $(\ell,0)$-JM
 partition $\nu$ with $\mathcal{R} \nu = \lambda$ (by Theorem \ref{deco}). From Theorem 
\ref{irreducible_nodes} we know that $\nu \in ladd_\ell$.
The fact that the crystals are isomorphic implies that $\hat{\varphi}_i(\nu) = \varphi$. By Theorem 
\ref{top_and_bottom_JM}, $\widehat{f}_i^{\varphi} \nu$ is another $(\ell,0)$-JM partition. 
Since regularization provides the isomorphism (see \cite{B}), we know that $\mathcal{R} \widehat{f}_i^{\varphi} \nu = \tilde{f}_i^{\varphi} \lambda$.
 Theorem $\ref{deco}$ then implies that $D^{\tilde{f}_i^{\varphi} \lambda} = S^{\widehat{f}_i^{\varphi} \nu}$
, since $S^{\widehat{f}_i^{\varphi} \nu}$ is irreducible by Theorem \ref{JM_irred}. 
Hence $\tilde{f}_i^{\varphi} \lambda$ is a weak $\ell$-partition. The proof of (2) is similar.

To prove (4), we must show that there does not exist an $(\ell,0)$-JM partition
$\mu$ in the regularization class of $\tilde{f}_i^k \lambda$. There exists an 
$(\ell,0)$-JM partition $\nu$ in $ladd_\ell$ so that $D^{\lambda} = S^{\nu}$. By Theorem 
\ref{top_and_bottom_JM}, $\widehat{f}_i^{k} \nu$ is not an $(\ell,0)$-JM partition. But by Theorem
 \ref{irreducible_nodes} we know all $(\ell,0)$-JM partitions occur in $ladd_\ell$. Also, 
only one element of $\mathcal{RC}(\widetilde{f}_i^k \lambda)$ occurs in $ladd_\ell$ and we
 know this is $\widehat{f}_i^k \nu$. Therefore no such $\mu$ can exist, so
 $\widetilde{f}_i^k \lambda$ is not a weak $\ell$-partition. (4) follows similarly.
\end{proof}

\begin{thma}
 One can also prove Theorem \ref{top_and_bottom_weak} via representation theory. For more details see the author's Ph.D. thesis \cite{Bphd}.
\end{thma}

\begin{section}{Acknowledgements}
 I would like to thank my advisor Monica Vazirani for her help and comments with this paper.
\end{section}

\bibliographystyle{amsalpha}

\begin{thebibliography}{A}

\bibitem{Bphd} C. Berg, Combinatorics of $(\ell,0)$-JM partitions, $\ell$-cores, the ladder crystal and the finite Hecke algebra, Ph. D. thesis. \textbf{arXiv} 0906.1559.
\bibitem{B} C. Berg, The ladder crystal, \textbf{Electronic Journal of Combinatorics}, to appear.
\bibitem{BV} C. Berg and M. Vazirani, ($\ell,0)$-Carter partitions, a generating function, and their crystal theoretic interpretation, \textbf{Electronic Journal of Combinatorics}, 15 (1) (2008).
 \bibitem{DJ}
R. Dipper and G. James, \textbf{Representations of Hecke algebras of general linear groups}, Proc. LMS (3), \textbf{52} (1986), 20-52
\bibitem{F}
M. Fayers, Irreducible Specht modules for Hecke Algebras of Type A, \textbf{Adv. Math.} \textbf{193} (2005) 438-452
\bibitem{J} G.D. James, The decomposition matrices of $GL_n(q)$ for $n\leq10$, \textbf{Proc. Lond. Math. Soc.} (3), \textbf{60} (1990), 225-264.
\bibitem{J2} G.D. James, On the Decomposition Matrices of the Symmetric Groups II, \textbf{J. of Algebra} \textbf{43} (1976), 45-54.
\bibitem{JK}
G.D. James and A.Kerber, \textbf{The Representation Theory of the Symmetric Group}, Encyclopedia of Mathematics, \textbf{16}, 1981.
\bibitem{JM}
G.D. James and A. Mathas, A q-analogue of the Jantzen-Schaper theorem, \textbf{Proc. Lond. Math. Soc.}, \textbf{74} (1997), 241-274. 
\bibitem{L}
S. Lyle, Some q-analogues of the Carter-Payne Theorem, \textbf{Journal f\"{u}r die reine und angewandte Mathematik} Volume 2007, Issue 608.
\bibitem{MM} K.C. Misra and T. Miwa, Crystal base for the basic representation of $U_q( \mathfrak{sl}_n )$, \textit{Commun. Math. Phys.} \textbf{134} (1990), 79-88.

\end{thebibliography}

\end{document}